\documentclass[10pt]{amsart}

\usepackage[cp1251]{inputenc}
\usepackage[T2A]{fontenc}
\usepackage{mathrsfs}
\usepackage{latexsym}
\usepackage{cite}
\usepackage{amsmath}
\usepackage{amsfonts}
\usepackage{amssymb}
\usepackage{amsthm}
\usepackage{euscript}
\usepackage[all]{xy}
\usepackage{delarray}
\usepackage{amscd}
\usepackage[active]{srcltx}
\usepackage[english]{babel}
\theoremstyle{plain}
\newtheorem{theorem}{Theorem}

\newtheorem{lemma}[theorem]{Lemma}
\newtheorem{proposition}[theorem]{Proposition}

\newtheorem{corollary}[theorem]{Corollary}

\theoremstyle{definition}

\theoremstyle{remark}

\newtheorem{remark}[theorem]{Remark}
\theoremstyle{plain}
\newtheorem*{theorem*}{Theorem}
\newtheorem*{lemma*}{Lemma}
\newtheorem*{proposition*}{Proposition}
\newtheorem*{statement*}{Statement}
\newtheorem*{corollary*}{Corollary}
\theoremstyle{definition}
\newtheorem*{definition*}{Definition}
\theoremstyle{remark}
\newtheorem*{notation*}{Notation}
\newtheorem*{remark*}{Remark}
\newtheorem*{example*}{Example}
\begin{document}
\title[Schr\"{o}dinger operators with singular matrix potentials]{Remarks on Schr\"{o}dinger operators with singular
matrix potentials}

\author[V. Mikhailets]{Vladimir Mikhailets}

\address{Institute of Mathematics \\
         National Academy of Science of Ukraine \\
         3 Tereshchenkivs'ka Str. \\
         01601 Kyiv-4 \\
         Ukraine}

\email{mikhailets@imath.kiev.ua}

\author[V. Molyboga]{Volodymyr Molyboga}

\address{Institute of Mathematics \\
         National Academy of Science of Ukraine \\
         3 Tereshchenkivs'ka Str. \\
         01601 Kyiv-4 \\
         Ukraine}

\email{molyboga@imath.kiev.ua}



\begin{abstract}
In this paper the asymmetric generalization of the Glazman--Povzner--Wienholtz theorem is proved for one-dimensional 
Schr\"{o}dinger operators with strongly singular matrix potentials from the space $H_{loc}^{-1}(\mathbb{R}, 
\mathbb{C}^{m\times m})$. This result is new in the scalar case as well. 
\end{abstract}

\maketitle

\section{Introduction and main results}
Let us consider in the complex separable Hilbert space of vector functions $L^{2}(\mathbb{R}, \mathbb{C}^{m})$, $m\in 
\mathbb{N}$ the operators generated by the formal differential expression: 
\begin{equation}\label{eq_10}
 \mathrm{l}[u]:=-u''+qu,\qquad u=(u_{1},\ldots,u_{m}).
\end{equation} 
where the matrix potential $q=\{q_{ij}\}_{i,j=1}^{m}$ belongs to the Sobolev negative class $H_{loc}^{-1}(\mathbb{R}, 
\mathbb{C}^{m\times m})$. 
Without loss of generality, we assume that the potential $q$ in \eqref{eq_10} may be presented in the form
\begin{equation*}\label{eq_12}
 q=Q'+s,\quad Q\in L_{loc}^{2}(\mathbb{R}, \mathbb{C}^{m\times m}),\;
 s\in L_{loc}^{1}(\mathbb{R}, \mathbb{C}^{m\times m}),
\end{equation*} 
where the derivative is understood in the sense of the distributions. Then the block Shin--Zettl matrices are defined: 
\begin{equation}\label{eq_14}
  A(x):=
  \begin{pmatrix}
     Q & I_{m} \\
    -Q^{2}+s & -Q  \
  \end{pmatrix}
   \in L_{loc}^{1}\left(\mathbb{R},\mathbb{C}^{2m\times 2m})\right),
\end{equation}
where $I_{m}$ is a unit $(m\times m)$-matrix. Similarly to the scalar case \cite{SvSh2003, GrMi2010} Shin--Zettl matrices 
define quasiderivatives \cite{MiSa2011}:
\begin{align*}\label{eq_16}
  u^{[0]}:=u,\qquad  u^{[1]}:=u'-Qu, \qquad
  u^{[2]}:=\left(u^{[1]}\right)'+Qu^{[1]}+\left(Q^{2}-s\right)u.
\end{align*}
Then formal differential equation \eqref{eq_10} is a quasidifferential one:
\begin{equation*}\label{eq_18}
  \mathrm{l}[u]:=-u^{[2]},\quad 
  \mathrm{Dom}(\mathrm{l}):=\left\{u\left|\,u,\,u^{[1]}\in 
  \mathrm{AC}_{loc}(\mathbb{R},\mathbb{C}^{m})\right.\right\},
\end{equation*}
where by $\mathrm{AC}_{loc}(\mathbb{R},\mathbb{C}^{m})$ we denote the class of locally absolutely continuous vector 
functions. This definition is motivated by the fact that
\begin{equation*}
 -u^{[2]}=-u''+qu
\end{equation*}
in the sense of distributions, i. e.,
\begin{equation*}\label{eq_20}
  \langle -u^{[2]},\varphi\rangle=\langle-u''+qu,\varphi\rangle,\qquad u\in \mathrm{Dom}(\mathrm{l}),\;
  \varphi\in \mathrm{C}_{0}^{\infty}(\mathbb{R},\mathbb{C}^{m}).
\end{equation*}
We say that function $u$ solves the Cauchy problem 
\begin{align}
 \mathrm{l}[u] & =f,\qquad f\in L_{loc}^{1}(\mathbb{R},\mathbb{C}^{m}), \label{eq_22.1} \\
 u(x_{0}) & =c_{0},\; u^{[1]}(x_{0})=c_{1},\qquad x_{0}\in \mathbb{R},\; c_{0},c_{1}\in \mathbb{C}^{m}, \label{eq_22.2}
\end{align}
if $u$ is the first coordinate of the vector function solving the Cauchy problem for the associated Cauchy problem with 
initial conditions \eqref{eq_22.2} 
\begin{equation}\label{eq_24}
\frac{d}{dx}
\begin{pmatrix}
 u \\
 u^{[1]}
\end{pmatrix}
=A(x)
\begin{pmatrix}
 u \\
 u^{[1]}
\end{pmatrix}
+
\begin{pmatrix}
 0 \\
 -f
\end{pmatrix}
.
\end{equation}

The existence and uniqueness theorem implies that the Cauchy problem for system \eqref{eq_24} has a unique solution (see 
\cite[Theorem~16.1]{Nai1969} and \cite[Theorem~2.1]{Wei1987}). Therefore our definition of a solution of the equation 
\eqref{eq_22.1} is correct.

Differential expression \eqref{eq_10} gives rise to the associated maximal and preminimal operators 
$\mathrm{L}$ and $\mathrm{L}_{00}$ in the Hilbert space $L^{2}(\mathbb{R}, \mathbb{C}^{m})$:
\begin{align*}
  \mathrm{L}u:=\mathrm{l}[u],\quad
   \mathrm{Dom}(\mathrm{L}):=\left\{u\in L^{2}(\mathbb{R},\mathbb{C}^{m})\,\left|\,u,\,u^{[1]}\in 
\mathrm{AC}_{loc}(\mathbb{R},\mathbb{C}^{m}),\;\mathrm{l}[u]\in L^{2}(\mathbb{R},\mathbb{C}^{m})\right.\right\},
\end{align*}
and 
\begin{align*}
  \mathrm{L}_{00}u:=\mathrm{l}[u], \qquad
   \mathrm{Dom}(\mathrm{L}_{00}):=\left\{u\in \mathrm{Dom}(\mathrm{L})\,
   \left|\,\mathrm{supp}\,u\Subset\mathbb{R}\right.\right\}. \hspace{80pt}
\end{align*}

The block Shin--Zettl matrix \eqref{eq_14} defines the Lagrange adjoint quasidifferential expression $\mathrm{l}^{+}$ 
in the following way:
\begin{align*}\label{eq_30}
  v^{\{0\}} & :=v,\qquad  v^{\{1\}}:=v'-Q^{\ast}v, \qquad
  v^{\{2\}}:=\left(v^{\{1\}}\right)'+Q^{\ast}v^{\{1\}}+\left((Q^{\ast})^{2}-s^{\ast}\right)v, \\
  \mathrm{l^{+}}[v] & :=-v^{\{2\}},\quad 
  \mathrm{Dom}(\mathrm{l}^{+}):=\left\{v\left|\,v,\,v^{\{1\}}\in 
  \mathrm{AC}_{loc}(\mathbb{R},\mathbb{C}^{m})\right.\right\},
\end{align*}
where the matrix $Q^{\ast}:=\overline{Q}^{T}$ is Hermitian conjugate to $Q$. The matrix $s^{\ast}$ has a similar meaning.

Quasidifferential expression $\mathrm{l}^{+}$ gives rise to the associated maximal and preminimal operators 
$\mathrm{L}^{+}$ and $\mathrm{L}_{00}^{+}$:
\begin{align*}
  \mathrm{L}^{+} & v:=\mathrm{l}^{+}[v], \\
   \mathrm{Dom}(\mathrm{L}^{+}) & :=\left\{v\in L^{2}(\mathbb{R},\mathbb{C}^{m})\,\left|\,v,\,v^{\{1\}}\in 
\mathrm{AC}_{loc}(\mathbb{R},\mathbb{C}^{m}),\;\mathrm{l}^{+}[v]\in L^{2}(\mathbb{R},\mathbb{C}^{m})\right.\right\},
\end{align*}
and 
\begin{align*}
  \mathrm{L}_{00}^{+}v:=\mathrm{l}^{+}[v], \qquad
   \mathrm{Dom}(\mathrm{L}_{00}^{+}):=\left\{v\in \mathrm{Dom}(\mathrm{L}^{+})\,
   \left|\,\mathrm{supp}\,v\Subset\mathbb{R}\right.\right\}. \hspace{80pt}
\end{align*}
Below we prove (Proposition~\ref{pr_10}) that preminimal operators $\mathrm{L}_{00}$, $\mathrm{L}_{00}^{+}$ 
are densely defined in the space $L^{2}(\mathbb{R},\mathbb{C}^{m})$ and have closures $\mathrm{L}_{0}$ and 
$\mathrm{L}_{0}^{+}$ which are called minimal operators. Maximal operators $\mathrm{L}$ and $\mathrm{L}^{+}$ are closed.

For the case of potential $q$ being a real-valued symmetric matrix such operators were considered earlier in 
\cite{MiSa2011}. Matrix Schr\"{o}dinger operators with strongly singular self-adjoint potentials of Miura class 
were investigated in detail in \cite{EcGsNcTs2013}. There one may find a more detailed review and a more extensive 
bibliography. For the scalar case of quasidifferential operators generated by Shin--Zettl matrices of general form 
one may find a review of results in \cite{EvMr1999}, see also \cite{Ztt1975, GrMiPn2013}.

Recall that an operator $A$ in the Hilbert space $H$ is called \textit{accretive} if 
\begin{equation*}
  \mathrm{Re}\,\langle Au,u\rangle_{H}\geq 0, \qquad u\in \mathrm{Dom}(A).
\end{equation*}
If in addition the left half-plane $\{\lambda\in \mathbb{C}\;|\;\mathrm{Re}\,\lambda<0\}$ 
belongs to the resolvent set of the operator $A$ then operator $A$ is called $m$-\textit{accretive} 
\cite{Kt1995, Schm2012}. This operator is also \textit{maximal accretive} in the sense that it has no accretive extensions in 
the space $H$. If operator $A$ is $m$-accretive then operator $-A$ generates a semigroup of contractions in the space $H$. 
Converse is also true. 

The main result of this paper is the non-symmetric generalization of the Glazman--Povzner--Wienholtz theorem for operators 
generated by differential expression \eqref{eq_10}.
\begin{theorem}\label{th_MnA}
The operator $\mathrm{L}_{0}$ is $m$-accretive if and only if preminimal operators $\mathrm{L}_{00}$ and 
$\mathrm{L}_{00}^{+}$ are accretive. In this case $\mathrm{L}_{0}=\mathrm{L}$.
\end{theorem}

Note that in this theorem we assume both preminimal operators $\mathrm{L}_{00}$ and $\mathrm{L}_{00}^{+}$ 
to be accretive. In the scalar case one of these operators being accretive implies that other is also accretive. 

\begin{corollary}[Cf. \cite{ClGs2003}]\label{cr_MnThA_10}
If matrix potential $q$ is self-adjoint: $Q=Q^{\ast}$ and $s=s^{\ast}$, then operator $\mathrm{L}_{0}$ is symmetric.
Moreover if operator $\mathrm{L}_{0}$ is bounded below then it is self-adjoint and $\mathrm{L}_{0}=\mathrm{L}$.
\end{corollary}
For $m=1$ this is known \cite[Remark~III.2]{AlKsMl2010}, see also \cite{MiMl6, HrMk2012, EcGsNcTs2012}.

\begin{remark}\label{rm_MnTh10}
If the complex matrices $Q$ and $s$ are symmetric, i. e., $Q=Q^{T}$, $s=s^{T}$, 
then Theorem~\ref{th_MnA} can be strengthened. 
As operator $\mathrm{L}_{00}$ is accretive, the operator $\mathrm{L}_{0}$ is maximal accretive 
and its residual spectrum is empty.
\end{remark}

In particular, this condition is satisfied in the scalar case, when $m=1$. 
In this case, the operators $\mathrm{L}_{00}$ and $\mathrm{L}_{00}^{+}$ obviously are accretive 
if the real part of the potential $q$ is positive in the sense of distributions.
This condition is equivalent to 
\begin{equation*}
 q=\mu+i\nu,
\end{equation*}
where  $\mu$ is a nonnegative Radon measure on a locally compact space $\mathbb{R}$ 
and $\nu$ is a real-valued distribution from $H_{loc}^{-1}(\mathbb{R}, \mathbb{C}^{m\times m})$.

The paper is organized as follows. 
In Section~\ref{sec_Prp} we state a list of the symbols used in the paper and thoroughly investigate the properties 
of the operators $\mathrm{L}$, $\mathrm{L}_{0}$ и $\mathrm{L}^{+}$, $\mathrm{L}_{0}^{+}$ (Proposition~\ref{pr_10}). 
Section~\ref{sec_Prf} contains proofs of the main Theorem~\ref{th_MnA}, 
Corollary~\ref{cr_MnThA_10} and Remark~\ref{rm_MnTh10}.


\section{Properties of the minimal and maximal operators}\label{sec_Prp}
In this paper, we use the following notation. 
We denote by $\left(\cdot\,,\cdot\right)_{\mathbb{C}^{m}}$ the inner product in the space $\mathbb{C}^{m}$:
\begin{equation*}
 \left(u,v\right)_{\mathbb{C}^{m}}:=\sum_{i=1}^{m}u_{i}\overline{v_{i}},\qquad u=(u_{1},\ldots,u_{m}),\; 
v=(v_{1},\ldots,v_{m})\in 
\mathbb{C}^{m}.
\end{equation*}
We denote by $\left\langle\cdot\,,\cdot\right\rangle_{L^{2}(\mathbb{R}, \mathbb{C}^{m})}$ 
the inner product in the Hilbert space of square-integrable vector functions $L^{2}(\mathbb{R}, \mathbb{C}^{m})$:
\begin{equation*}
 \left\langle u,v\right\rangle_{L^{2}(\mathbb{R}, \mathbb{C}^{m})}:=\int_{\mathbb{R}}(u,v)_{\mathbb{C}^{m}}d\,x
\end{equation*}
For an arbitrary matrix $A=\{a_{ij}\}_{i,j=1}^{m}\in \mathbb{C}^{m\times m}$ 
we denote the transposed matrix by $A^{T}=\{a_{ij}^{T}\}_{i,j=1}^{m}$ 
and Hermitian conjugate matrix by $A^{\ast}=\{a_{ij}^{\ast}\}_{i,j=1}^{m}$: 
$a_{ij}^{\ast}=\overline{a_{ji}}$. 
For an arbitrary complex number $a\in \mathbb{C}$ we denote the corresponding complex conjugate number by 
$\overline{a}$. 

We say that matrix function $A(x)=\{a_{ij}(x)\}_{i,j=1}^{m}$ belongs to the space 
$L_{loc}^{p}(\mathbb{R}, \mathbb{C}^{m\times m})$, 
if each element of this matrix $a_{ij}(x)$ belongs to the space 
$L_{loc}^{p}(\mathbb{R}, \mathbb{C})$, $p\in [1,\infty)$. 

J. Weidmann \cite{Wei1987} previously studied in detail the quasidifferential matrix-valued Sturm-Liouville operators 
generated by quasidifferential expressions $\tau$, 
\begin{align*}\label{eq_34}
 \tau[u] & := -(u'-Qu)'-Q^{\ast}(u'-Qu)-(Q^{\ast}Q-s)u, \\
 Q & \in L_{loc}^{2}(\mathbb{R},\mathbb{C}^{m\times m}),\;s\in L_{loc}^{1}(\mathbb{R},\mathbb{C}^{m\times m}),\; s=s^{\ast}.
\end{align*}
In this case preminimal operators generated by quasidifferential expressions $\tau$ are symmetric
\cite[Theorem~3.1]{Wei1987}.

Obviously, if matrices $Q=Q^{\ast}$ and $s=s^{\ast}$ are self-adjoint 
then operators generated by quasidifferential expressions $\tau$ 
and operators generated by quasidifferential expressions $\mathrm{l}$ and $\mathrm{l}^{+}$ coincide.

The following properties of the operators $\mathrm{L}$, $\mathrm{L}_{0}$, $\mathrm{L}_{00}$ и $\mathrm{L}^{+}$, 
$\mathrm{L}_{0}^{+}$, $\mathrm{L}_{00}^{+}$ we state without proof, 
because they are proved in the same way as the properties of operators generated by quasidifferential expressions~$\tau$
\cite{Wei1987}.

\begin{lemma}\label{lm_Prp10}
For arbitrary vector functions $u\in \mathrm{Dom}(\mathrm{L})$, $v\in \mathrm{Dom}(\mathrm{L}^{+})$ 
and arbitrary finite interval $[a,b]$ we have
\begin{equation*}\label{eq_36}
 \int_{a}^{b}\left(l[u],v\right)_{\mathbb{C}^{m}}d\,x-\int_{a}^{b}\left(u,\mathrm{l}^{+}[v]\right)_{\mathbb{C}^{m}}d\,x
 =[u,v]_{ a } ^ { b }, 
\end{equation*} 
where 
\begin{align*}
  [u,v](t)\equiv [u,v] & :=\left(u,v^{\{1\}}\right)_{\mathbb{C}^{m}}-\left(u^{[1]},v\right)_{\mathbb{C}^{m}}, \\
  [u,v]_{a}^{b} & :=[u,v](b)-[u,v](a),\quad -\infty\leq a\leq b\leq \infty.
\end{align*}  
\end{lemma}

\begin{lemma}\label{lm_Prp12}
For arbitrary vector functions $u\in \mathrm{Dom}(\mathrm{L})$ and $v\in \mathrm{Dom}(\mathrm{L}^{+})$ 
the following limits exist and are finite:
\begin{equation*}\label{eq_40}
  [u,v](-\infty):=\lim_{t\rightarrow-\infty}[u,v](t),\qquad [u,v](\infty):=\lim_{t\rightarrow\infty}[u,v](t).
\end{equation*}
\end{lemma}

\begin{lemma}[Generalized Lagrange identity]\label{lm_Prp14}
For arbitrary vector functions $u\in \mathrm{Dom}(\mathrm{L})$ and $v\in \mathrm{Dom}(\mathrm{L}^{+})$ 
the following relation holds:
\begin{equation*}\label{eq_42}
 \int_{-\infty}^{\infty}\left(l[u],v\right)_{\mathbb{C}^{m}}d\,x 
-\int_{-\infty}^{\infty}\left(l[u],v\right)_{\mathbb{C}^{m}}d\,x=[u,v]_{ -\infty}^{\infty}.
\end{equation*}
\end{lemma}

\begin{proposition}\label{pr_10}
The operators $\mathrm{L}$, $\mathrm{L}_{00}$ and $\mathrm{L}^{+}$, $\mathrm{L}_{00}^{+}$ have the following properties:
\begin{itemize}
  \item [$1^{0}$.] Operators $\mathrm{L}_{00}$ and $\mathrm{L}_{00}^{+}$ are densely defined in the Hilbert space  
$L^{2}(\mathbb{R},\mathbb{C}^{m})$.
  \item [$2^{0}$.] The equalities 
\begin{equation*}
  \left(\mathrm{L}_{00}\right)^{\ast}=\mathrm{L}^{+},\qquad \left(\mathrm{L}_{00}^{+}\right)^{\ast}=\mathrm{L}
\end{equation*}
hold. 
In particular, operators $\mathrm{L}$, $\mathrm{L}^{+}$ are closed and operators $\mathrm{L}_{00}$, $\mathrm{L}_{00}^{+}$ are closable.
  \item [$3^{0}$.] Domains of operators $\mathrm{L}_{0}$, $\mathrm{L}_{0}^{+}$ may be described in the following way:
\begin{align*}
  \mathrm{Dom}(\mathrm{L}_{0}) & =\left\{u\in \mathrm{Dom}(\mathrm{L}) \left|\,[u,v]_{-\infty}^{\infty}=0\quad \forall
v\in  \mathrm{Dom}(\mathrm{L}^{+})\right.\right\}, \\
  \mathrm{Dom}(\mathrm{L}_{0}^{+}) & =\left\{v\in \mathrm{Dom}(\mathrm{L}^{+}) \left|\,[u,v]_{-\infty}^{\infty}=0\quad
\forall u\in  \mathrm{Dom}(\mathrm{L})\right.\right\}.
\end{align*}
  \item [$4^{0}$.] The following inclusions take place:
\begin{equation*}
  \mathrm{Dom}(\mathrm{L})\subset H_{loc}^{1}(\mathbb{R},\mathbb{C}^{m}),\;
  \mathrm{Dom}(\mathrm{L}^{+})\subset H_{loc}^{1}(\mathbb{R},\mathbb{C}^{m}).
\end{equation*}
\end{itemize}
\end{proposition}

For the case $m=1$ the results of this section are established in \cite{MiMl2013}.



\section{Proofs}\label{sec_Prf}
The following lemma is proved by direct calculation.
\begin{lemma}\label{lm_Prf10}
For arbitrary vector functions $u\in \mathrm{Dom}(\mathrm{L})$, $v\in \mathrm{Dom}(\mathrm{L}^{+})$ 
and functions  $\varphi\in C_{0}^{\infty}(\mathbb{R},\mathbb{C})$ we have 
\begin{align*}
 i)\; & \mathrm{l}[\varphi I_{m}u]=\varphi I_{m}\mathrm{l}[u]-\varphi''I_{m}u-2\varphi'I_{m}u',\quad 
 \varphi I_{m}u\in \mathrm{Dom}(\mathrm{L}_{00}); \\
 ii)\; & \mathrm{l}^{+}[\varphi I_{m}v]=\varphi I_{m}\mathrm{l}^{+}[v]-\varphi''I_{m}v-2\varphi'I_{m}v',\quad 
 \varphi I_{m}v\in \mathrm{Dom}(\mathrm{L}_{00}^{+}). \hspace{95pt}
\end{align*}
\end{lemma}

\begin{proof}[Proof of Theorem~\ref{th_MnA}]
\textit{Sufficiency.} 
Due to the assumptions of theorem the minimal operators $\mathrm{L}_{0}$ and $\mathrm{L}_{0}^{+}$ are accretive. 
Without loss of generality we assume that the following inequalities hold:
\begin{equation*}\label{eq_Prf10}
  \mathrm{Re}\,\left\langle \mathrm{L}_{0}u,u\right\rangle _{L^{2}(\mathbb{R},\mathbb{C}^{m})}\geq 
\left\langle u,u\right\rangle _{L^{2}(\mathbb{R},\mathbb{C}^{m})},\qquad u\in \mathrm{Dom}(\mathrm{L}_{0}),
\end{equation*}
and 
\begin{equation}\label{eq_Prf12}
  \mathrm{Re}\,\left\langle \mathrm{L}_{0}^{+}v,v\right\rangle _{L^{2}(\mathbb{R},\mathbb{C}^{m})}\geq 
\left\langle v,v\right\rangle _{L^{2}(\mathbb{R},\mathbb{C}^{m})},\qquad v\in \mathrm{Dom}(\mathrm{L}_{0}^{+}).
\end{equation}

To prove the minimal operator $\mathrm{L}_{0}$ to be $m$-accretive one suffices to show that the kernel of operator 
$\mathrm{L}^{+}$ contains only the zero element.

Let $v$ be a solution to the equation 
\begin{equation*}\label{eq_Prf14}
  \mathrm{L}^{+}v=0.
\end{equation*}
We will show that $v\equiv 0$.

For an arbitrary function $\varphi\in \mathrm{C}_{0}^{\infty}(\mathbb{R},\mathbb{R})$ due to Lemma~\ref{lm_Prf10} 
we have $\varphi I_{m} v\in\mathrm{Dom}(\mathrm{L}_{00}^{+})$. 
Therefore, taking into account that $\mathrm{l}^{+}[v]=0$, after some simple calculations we obtain:
\begin{equation}\label{eq_Prf16}
  \left\langle\mathrm{L}_{0}^{+}\varphi I_{m} v,\varphi I_{m} v\right\rangle_{L^{2}(\mathbb{R},\mathbb{C}^{m})} 
  =\int_{\mathbb{R}}(\varphi')^{2}(v,v)_{\mathbb{C}^{m}}d\,x
  +\int_{\mathbb{R}}\varphi\varphi'\left((v,v')_{\mathbb{C}^{m}}-(v',v)_{\mathbb{C}^{m}}\right) d\,x.
\end{equation}
As 
\begin{equation*}
 \mathrm{Re}\int_{\mathbb{R}}\varphi\varphi'\left((v,v')_{\mathbb{C}^{m}}-(v',v)_{\mathbb{C}^{m}}\right) d\,x=0,
\end{equation*}
from \eqref{eq_Prf16} taking into account \eqref{eq_Prf12} we receive:
\begin{equation}\label{eq_Prf18}
 \int_{\mathbb{R}}(\varphi')^{2}(v,v)_{\mathbb{C}^{m}}d\,x\geq \int_{\mathbb{R}}(\varphi)^{2}(v,v)_{\mathbb{C}^{m}}d\,x
 \qquad \forall\varphi \in \mathrm{C}_{0}^{\infty}(\mathbb{R},\mathbb{R}).
\end{equation} 
Furthermore, let us take a sequence of functions $\{\varphi_{n}\}_{n\in \mathbb{N}}$ which has the following properties:
\begin{itemize}
 \item [i)] $\varphi_{n}\in C_{0}^{\infty}(\mathbb{R},\mathbb{R})$; 
 \item [ii)] $\mathrm{supp}\,\varphi_{n}\subset [-n-1,n+1]$;
 \item [iii)] $\varphi_{n}(x)=1$, $x\in [-n,n]$;
 \item [iv)] $|\varphi_{n}'(x)|\leq C$ where $C>0$ is an absolute constant.
\end{itemize}
Substituting in \eqref{eq_Prf18} we get 
\begin{equation*}
 \int_{-n}^{n}(v,v)_{\mathbb{C}^{m}}d\,x\leq \int_{\mathbb{R}}\varphi_{n}^{2}(v,v)_{\mathbb{C}^{m}}d\,x\leq 
 \int_{\mathbb{R}}(\varphi_{n}')^{2}(v,v)_{\mathbb{C}^{m}}d\,x
 \leq C^{2}\int_{n\leq|x|\leq n+1}\limits(v,v)_{\mathbb{C}^{m}}d\,x,
\end{equation*}
i. e. 
\begin{equation}\label{eq_Prf20}
 \int_{-n}^{n}(v,v)_{\mathbb{C}^{m}}d\,x\leq C^{2}\int_{n\leq|x|\leq n+1}\limits(v,v)_{\mathbb{C}^{m}}d\,x.
\end{equation}
As $v\in L^{2}(\mathbb{R},\mathbb{C}^{m})$ passing in \eqref{eq_Prf20} to the limit as $n\rightarrow\infty$, 
we receive $v\equiv 0$.

Thus we have proved that operator $\mathrm{L}_{0}$ is $m$-accretive.

In a similar way one may prove that operator  $\mathrm{L}_{0}^{+}$  is $m$-accretive. 
Then taking into account that an adjoint operator to the $m$-accretive operator is $m$-accretive  \cite[Proposition~3.20]{Schm2012} from the property~$2^{0}$ of Proposition~\ref{pr_10} 
we get that the maximal operator $\mathrm{L}$ is also $m$-accretive. 
By the definition of the maximal accretivity and \cite[Предложение~3.24]{Schm2012} we have $\mathrm{L}_{0}=\mathrm{L}$ 
as $\mathrm{L}_{0}\subset\mathrm{L}$.
Sufficiency is proved. 

\textit{Necessity.} Let us suppose that the operator $\mathrm{L}_{0}$ is $m$-accretive. Then taking into account that an 
adjoint operator to the $m$-accretive operator is $m$-accretive \cite[Proposition~3.20]{Schm2012} from the property~$2^{0}$ 
of Proposition~\ref{pr_10} we get that the operator $\mathrm{L}_{0}^{+}$ is $m$-accretive. Therefore the operators 
$\mathrm{L}_{00}$ and $\mathrm{L}_{00}^{+}$ are accretive. Necessity is proved.

Theorem is proved completely.
\end{proof}

\begin{proof}[Proof of Corollary~\ref{cr_MnThA_10}]
One only needs to note that in the case of self-adjoint potential $q$ preminimal operators $\mathrm{L}_{00}$ and 
$\mathrm{L}_{00}^{+}$ coincide and due to property $2^{0}$ of Proposition~\ref{pr_10} 
(see also \cite[Theorem~3.1]{Wei1987}) are symmetric.
\end{proof}

\begin{proof}[Proof of Remark~\ref{rm_MnTh10}]
Note that in the case of complex symmetric matrix potentials, we have:
\begin{equation*}
 Q^{*}=\overline{Q}=\{\overline{Q}_{ij}\}_{i,j=1}^{m},\qquad s^{*}=\overline{s}=\{\overline{s}_{ij}\}_{i,j=1}^{m}.
\end{equation*}
Then domains of preminimal operators $\mathrm{L}_{00}$ и $\mathrm{L}_{00}^{+}$ are related by 
\begin{equation*}
 u\in \mathrm{Dom}(\mathrm{L}_{00})\Leftrightarrow \overline{u}\in \mathrm{Dom}(\mathrm{L}_{00}^{+}).
\end{equation*}
Therefore the accretivity of the operator $\mathrm{L}_{00}$ implies the accretivity of the operator $\mathrm{L}_{00}^{+}$ 
and vice versa.

Moreover, let $\mathrm{J}$ be an antilinear  operator of complex conjugation. 
Then one may easy verify that the following inclusion takes place: 
\begin{equation*}
 \mathrm{J}\mathrm{L}_{0}\mathrm{J}=\mathrm{L}_{0}^{+}\subset \mathrm{L}^{+}=\mathrm{L}_{0}^{\ast}, 
\end{equation*}
that is, the operator $\mathrm{L}_{0}$ is $\mathrm{J}$-symmetric~\cite{Glz1966}. 
If operators $\mathrm{L}_{00}$ are accretive, then due to Theorem~\ref{th_MnA} and property~$2^{0}$ 
of Proposition~\ref{pr_10} the operator $\mathrm{L}_{0}$ is $\mathrm{J}$-self-adjoint:
\begin{equation*}
 \mathrm{J}\mathrm{L}_{0}\mathrm{J}=\mathrm{L}_{0}^{\ast}.
\end{equation*}
Therefore its residual spectrum is empty.
\end{proof}

\vspace{25pt}

\textit{Acknowledgment}. The first author was partially supported by the grant no. 03-01-12 of National Academy of Sciences 
of Ukraine (under the joint Ukrainian--Russian project of NAS of Ukraine and  Syberian Branch of Russian Academy of Sciences) 
and the second author was partially supported by the grant no. 01-01-12 of National Academy of Sciences of Ukraine (under the 
joint Ukrainian--Russian project of NAS of Ukraine and Russian Foundation of Basic Research).




\end{document}